\documentclass{article}
\usepackage{amsmath, amsfonts, amsthm, amssymb}
\usepackage{graphicx}
\usepackage{float}
\usepackage{verbatim}

\numberwithin{equation}{section}
\newtheorem{thm}{Theorem}[section]
\newtheorem{lma}[thm]{Lemma}
\newtheorem{cor}[thm]{Corollary}
\newtheorem{defn}[thm]{Definition}

\newtheorem{prop}[thm]{Proposition}

\newtheorem{ques}[thm]{Question}

\renewcommand{\geq}{\geqslant}
\renewcommand{\leq}{\leqslant}
\renewcommand{\H}{\text{H}}
\renewcommand{\P}{\text{P}}

\title{Inhomogeneous self-similar sets and box dimensions}

\author{Jonathan M. Fraser\\ \\
\emph{Mathematical Institute, University of St Andrews, North Haugh,}\\ \emph{St Andrews, Fife, KY16 9SS, Scotland}\\ \emph{e-mail: jmf32@st-andrews.ac.uk}}

\begin{document}
\maketitle

\begin{abstract}
We investigate the box dimensions of inhomogeneous self-similar sets.  Firstly, we extend some results of Olsen and Snigireva by computing the upper box dimensions assuming some mild separation conditions.  Secondly, we investigate the more difficult problem of computing the lower box dimension.  We give some non-trivial bounds and provide examples to show that lower box dimension behaves much more strangely than the upper box dimension, Hausdorff dimension and packing dimension.
\end{abstract}

\section{Introduction}

In this paper we investigate the dimensions of inhomogeneous attractors.  If a dimension function is countably stable (the dimension of a countable union of sets is equal to the supremum of the individual dimensions), then the dimension is easy to compute.  In particular, the dimension is the maximum of the dimension of the corresponding homogeneous attractor and the dimension of the condensation set.  However, if a dimension function is not countably stable, then the problem is more difficult.  As such we investigate the (not countably stable) \emph{box dimensions} of inhomogeneous self-similar sets.   We extend some results of Olsen and Snigireva \cite{olseninhom, ninaphd} by computing the upper box dimensions assuming some mild separation conditions.  We show that in our setting the upper box dimension behaves in the same way as the countably stable dimensions.  Secondly, we investigate the more difficult problem of computing the lower box dimension.  We give some non-trivial bounds on the lower box dimension and prove that it does not behave as the other dimensions.  In particular, the lower box dimension is not in general the maximum of the lower box dimensions of the homogeneous self-similar set and the condensation set.  We introduce a quantity which we call the \emph{covering regularity exponent} which is designed to give information about the oscillatory behaviour of the covering function $N_\delta$ and use it to study the lower box dimensions.  We believe the covering regularity exponent will be a useful quantity in other circumstances where one needs finer information about the asymptotic properties of $N_\delta$, or indeed other function where the asymptotic oscillations are important.

\subsection{Inhomogeneous attractors}

Let $(X,d)$ be a compact metric space.  An iterated function system (IFS) is a finite collection $\mathbb{I}=\{ S_i \}_{i=1}^{N}$ of contracting self maps on $X$.  It is a fundamental result in fractal geometry (see \cite{hutchinson}) that for every IFS there exists a unique non-empty compact set,  $F$, called the attractor, which satisfies
\begin{equation} \label{hom}
F = \bigcup_{i=1}^{N} S_i (F).
\end{equation}
We call such attractors \emph{homogeneous} attractors.  Now fix a compact set $C \subseteq X$, sometimes called the \emph{condensation set}.  Analogous to above, there is a unique non-empty compact set, $F_C$, satisfying
\begin{equation} \label{inhom}
F_C =  \bigcup_{i=1}^{N} S_i (F_C) \ \cup \ C,
\end{equation}
which we refer to as an \emph{inhomogeneous} attractor (with condensation $C$).  Note that homogeneous attractors are inhomogeneous attractors with condensation equal to the empty set.  From now on we will assume that the condensation set is non-empty.  Inhomogeneous attractors were introduced and studied in \cite{barndemko} and are also discussed in detail in \cite{superfractals} where, among other things, Barnsley gives applications of these schemes to image compression.  Define the \emph{orbital set}, $\mathcal{O}$, by
\[
\mathcal{O} \ = \ C \ \cup \  \bigcup_{k \in \mathbb{N}} \ \  \bigcup_{i_1, \dots, i_k \in \{1, \dots, N\}} S_{i_1} \circ \cdots \circ S_{i_k}(C),
\]
i.e., the union of the condensation set, $C$, with all images of $C$ under compositions of maps in the IFS.  The term \emph{orbital set} was introduced in \cite{superfractals} and it it turns out that this set plays an important role in the structure of inhomogeneous attractors and, in particular,
\begin{equation} \label{structure}
F_C \ = \  F_\emptyset \cup \mathcal{O} \ = \  \overline{\mathcal{O}},
\end{equation}
where $F_\emptyset$ is the \emph{homogeneous} attractor of the IFS, $\mathbb{I}$.

\begin{figure}[H]
	\centering
	\includegraphics[width=70mm]{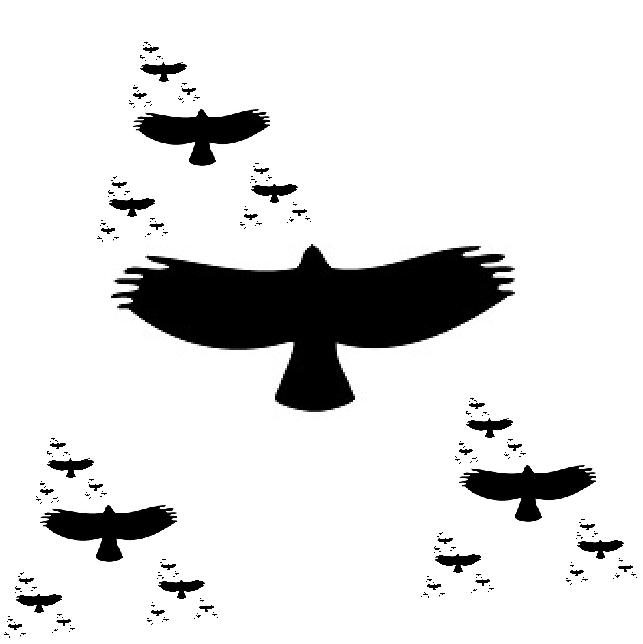}
\qquad \qquad
	\includegraphics[width=70mm]{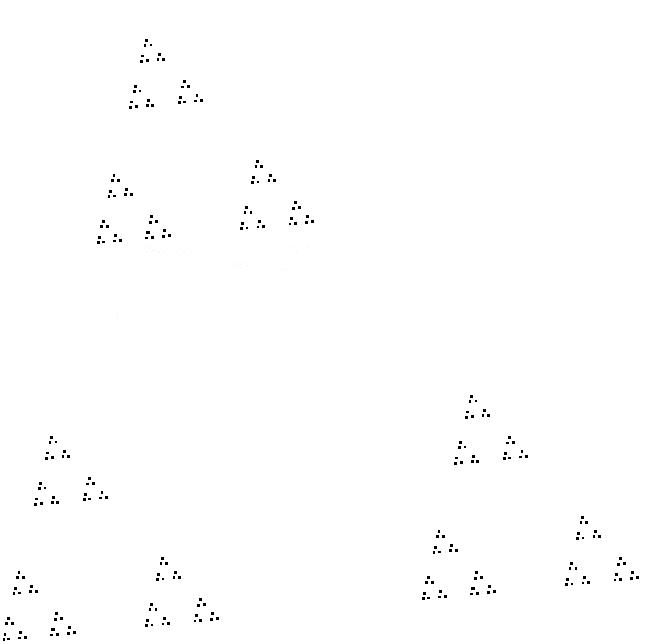}
	\caption{\emph{A flock of birds from above} (left).  The `flock' is represented by an inhomogeneous self-similar set.  The large bird in the middle is the condensation and there are 3 similarity mappings in the IFS all with contraction ratio 1/3.  The corresponding homogeneous attractor is shown on the right.  This is a totally disconnected self-similar set with Hausdorff and box dimension equal to 1.
}
\end{figure}

The relationship (\ref{structure}) was proved in \cite[Lemma 3.9]{ninaphd} in the case where $X$ is a compact subset of $\mathbb{R}^d$ and the maps are similarities.  We note here that their arguments easily generalise to obtain the general case stated above.  Writing $\dim_\text{\H}$ and $\dim_\text{\P}$ for Hausdorff and packing dimension respectively, it follows immediately from (\ref{structure}) that
\[
\dim_\H F_C = \max \{ \dim_\H  F_\emptyset, \ \dim_\H C\} \qquad \text{and} \qquad \dim_\P F_C = \max \{ \dim_\P F_\emptyset, \ \dim_\P C\}
\]
since Hausdorff and packing dimension are countably stable.  Indeed such a relation holds for any definition of dimension which is countably stable, for example, \emph{modified} upper and lower box dimension, see \cite[Section 3.3]{falconer}.  However, upper and lower box dimension are \emph{not} countably stable and in fact lower box dimension is not even finitely stable, and so studying the lower and upper box dimensions of inhomogeneous attractors is a more subtle problem.  In \cite[Corollary 2.6]{olseninhom} and \cite[Theorem 3.10 (2)]{ninaphd} it was proved that if $X \subset \mathbb{R}^d$; each of the $S_i$ are similarities; and the sets $S_1 (F_C), \dots, S_N(F_C),  C$ are pairwise disjoint, then
\[
\overline{\dim}_\text{B} F_C = \max \{ \overline{\dim}_\text{B}  F_\emptyset, \ \overline{\dim}_\text{B} C\}.
\]
The authors then asked the following question, see \cite[Question 2.7]{olseninhom} and \cite[Question 3.12]{ninaphd}.
\begin{ques}\label{quest}
Does the above formula for upper box dimension remain true if we relax the separation conditions to only the inhomogeneous open set condition (IOSC)?
\end{ques}

We give an affirmative answer to this question and furthermore prove that it holds assuming only that the IFS, $\mathbb{I}$, satisfies the strong open set condition (which is equivalent to the open set condition if $X \subset \mathbb{R}^d$), see Corollary \ref{ques}, and even without assuming any separation conditions it holds \emph{generically}, see Corollary \ref{generic}.  We remark here that the definitions of the IOSC given in \cite{olseninhom, ninaphd} are slightly different.  Rather than give both of the technical definitions we simply remark that we are able to answer Question \ref{quest} using significantly weaker separation conditions than either version of the IOSC.  In particular, the condensation set can have arbitrary overlaps with the basic sets in the construction of the homogeneous attractor.
\\ \\
In \cite{olseninhom, ninaphd} the authors also point out that they are not aware if the corresponding formula holds for lower box dimension.  The following question is asked in \cite[Question 3.11]{ninaphd}.
\begin{ques}\label{quest2}
If $X \subset  \mathbb{R}^d$, each of the $S_i$ are similarities and the sets $S_1 (F_C), \dots, S_N(F_C),  C$ are pairwise disjoint, then is it true that
\[
\underline{\dim}_\text{\emph{B}} F_C = \max \{ \underline{\dim}_\text{\emph{B}}  F_\emptyset, \ \underline{\dim}_\text{\emph{B}} C\}?
\]
\end{ques}
We prove that the answer to this question is \emph{no}, see Theorem \ref{lowerbound} and Proposition \ref{exlem} (2).  We also give some sufficient conditions for the answer to be \emph{yes}, see Corollary \ref{trivial bounds}, Corollary \ref{bzerocor}, Theorem \ref{upperbound} and Proposition \ref{exlem} (1).

\subsection{Basic definitions and notation}

In this section we recall some basic definitions and fix some notation needed to state our results.  The following separation condition is fundamental in the theory of IFSs.

\begin{defn}
An IFS, $\{S_{1}, \dots,  S_{N}\}$, with attractor $F$ satisfies the \emph{strong open set condition (SOSC)}, if there exists a non-empty open set, $U$, with $F \cap U \neq \emptyset$ and such that
\[
\bigcup_{i=1}^N S_{i}(U) \subseteq U
\]
with the union disjoint.
\end{defn}

A celebrated result of Schief \cite{schief} is that the SOSC is equivalent to the weaker \emph{open set condition} (OSC) if $X \subset  \mathbb{R}^d$ and the maps in the IFS are similarities.  The OSC is the same as the SOSC but without the requirement that $F \cap U \neq \emptyset$.  We adapt the SOSC to the case of inhomogeneous attractors in the following way.

\begin{defn}
An IFS, $\{ S_{1}, \dots,  S_{N}\}$, together with a compact set $C \subseteq X$, satisfies the \emph{condensation open set condition (COSC)}, if the IFS, $\{ S_{1}, \dots,  S_{N}\}$, satisfies the SOSC and the open set, $U$, can be chosen such that $C \subseteq \overline{U}$.
\end{defn}

The COSC will only be used to obtain one of our results, Theorem \ref{lowerbound}.  
\\ \\
Recall that a map $S:X \to X$ is called a similarity if, for all $x,y \in X$, we have $d(S(x),S(y)) = c \, d(x,y)$ for some constant $c \in (0,1)$.  The constant $c$ is called the \emph{Lipschitz constant} and for a similarity, $S$, we will write $\text{Lip}(S)$ to denote the Lipschitz constant for $S$.  Given an IFS, $\mathbb{I}=\{ S_{1}, \dots,  S_{N}\}$, consisting of similarities, the \emph{similarity dimension} of the homogeneous attractor of $\mathbb{I}$ is defined to be the unique solution to Hutchison's formula
\begin{equation} \label{hutch}
\sum_{i=1}^{N} \text{Lip}(S_i)^s = 1.
\end{equation}
It is well-known that if such an IFS satisfies the SOSC, then the similarity dimension equals the Hausdorff, packing and box dimension of the homogeneous attractor, see \cite{schief2}, or for the Euclidean case see \cite{hutchinson} or \cite[Section 9.3]{falconer}.  We will now recall the definition of box dimension.  For a non-empty subset $F \subseteq X$ and some $\delta>0$, let $N_\delta(F)$ be the minimum number of sets of diameter $\delta$ required to cover $F$.  The lower and upper box dimension of $F$ are defined by
\[
\underline{\dim}_\text{B} F = \liminf_{\delta \to 0} \, \frac{\log N_\delta (F)}{-\log \delta}
\qquad
\text{and}
\qquad
\overline{\dim}_\text{B} F = \limsup_{\delta \to 0} \,  \frac{\log N_\delta (F)}{-\log \delta},
\]
respectively.  If $\underline{\dim}_\text{B} F = \overline{\dim}_\text{B} F$, then we call the common value the \emph{box dimension} of $F$ and denote it by $\dim_\text{B} F$.  What we call the box dimension is also sometimes referred to as the box-counting dimension, entropy dimension or Minkowski dimension.  For a non-empty subset $F \subseteq X$ we note the following relationships between the dimensions discussed above
\[
\begin{array}{ccccc}
& &                                                                              \dim_\text{\P} F                                    & &  \\
 &                                 \rotatebox[origin=c]{45}{$\leq$}          & &                \rotatebox[origin=c]{315}{$\leq$} &  \\
 \dim_\text{\H} F                                             & & & &                         \overline{\dim}_\text{B} F. \\
 &                                 \rotatebox[origin=c]{315}{$\leq$}       & &                  \rotatebox[origin=c]{45}{$\leq$} &  \\
& &                                                                        \underline{\dim}_\text{B} F                           & &  
\end{array}
\]
Furthermore,  if $F$ is a \emph{homogeneous} self-similar set, then we have equality of these four dimensions, regardless of separation conditions, see \cite{implicit} or \cite[Corollary 3.3]{techniques}. For more details on the basic properties of box dimension and its interplay with the Hausdorff and packing dimensions, the reader is referred to \cite[Chapter 3]{falconer}.

\section{Results}

In this section we will state our main results.  Fix an IFS $\mathbb{I}=\{ S_{1}, \dots,  S_{N}\}$ where each $S_i$ is a similarity on $(X,d)$, fix a non-empty compact condensation set $C \subseteq X$ and let $s$ denote the similarity dimension of $F_\emptyset$.  Our results concerning upper box dimension will be given in Section \ref{uppersect} and those concerning lower box dimension will be given in Section \ref{lowersect}.
We will write $B(x,r)$ to denote the open ball of radius $r$ centered at $x$.

\subsection{Upper box dimension} \label{uppersect}

In this section we significantly generalise the results in \cite{olseninhom, ninaphd} concerning upper box dimension, which were obtained as Corollaries to results on the $L^q$-dimensions of inhomogeneous self-similar measures.  Our proofs are direct and deal only with sets.  Our first result bounds the upper box dimension of an inhomogeneous self-similar set, without assuming any separation conditions.

\begin{thm} \label{main}
We have
\[
\max \{ \overline{\dim}_\text{\emph{B}} F_\emptyset, \ \overline{\dim}_\text{\emph{B}} C\} \ \leq \ \overline{\dim}_\text{\emph{B}} F_C  \ \leq \ \max \{ s, \ \overline{\dim}_\text{\emph{B}} C\}.
\]
\end{thm}

Although the bounds given in Theorem \ref{main} are not tight in general, we can apply them in two useful situations to obtain an exact result.  The following Corollary answers Question \ref{quest} in the affirmative and, in fact, proves something stronger in that the separation conditions can be severely weakened and we can work in an arbitrary compact metric space.

\begin{cor} \label{ques}
Suppose that the IFS, $\mathbb{I}$, satisfies the SOSC.  Then
\[
\overline{\dim}_\text{\emph{B}} F_C = \max \{ \overline{\dim}_\text{\emph{B}} F_\emptyset, \ \overline{\dim}_\text{\emph{B}} C\}.
\]
\end{cor}

\begin{proof}
This follows immediately from Theorem \ref{main} since if $\mathbb{I}$ satisfies the SOSC, then $s = \overline{\dim}_\text{B} F_\emptyset$, see \cite{schief2}.
\end{proof}

Of course, if $X \subseteq \mathbb{R}^d$ then the SOSC is equivalent to the OSC.  We can also obtain an exact result in a generic sense.

\begin{cor} \label{generic}
Let $d \in \mathbb{N}$ and fix linear contracting similarities, $\{ T_1, \dots, T_N\}$, each mapping $\mathbb{R}^d$ to itself, and assume that $\text{\emph{Lip}}(T_i)<1/2$ for all $i$ and fix a compact condensation set $C \subset \mathbb{R}^d$.  For $\mathbf{t} = (t_1, \dots, t_N) \in \times_{i=1}^N \mathbb{R}^d$, let $F_{\mathbf{t}, \emptyset}$ denote the homogeneous attractor satisfying
\[
F_{\mathbf{t}, \emptyset} = \bigcup_{i=1}^{N} \Big(  T_i (F_{\mathbf{t}, \emptyset})+t_i \Big) 
\]
and let $F_{\mathbf{t}, C}$ denote the inhomogeneous attractor satisfying
\[
F_{\mathbf{t}, C} \ = \ \bigcup_{i=1}^{N} \Big( T_i (F_{\mathbf{t},C})+t_i \Big) \ \cup \ C.
\]
Then, writing $\mathcal{L}^{dN}$ for the $N$-fold product of $d$-dimensional Lebesgue measure, we have
\[
\overline{\dim}_\text{\emph{B}} F_{\mathbf{t},C} = \max \{ \overline{\dim}_\text{\emph{B}} F_{\mathbf{t},\emptyset}, \ \overline{\dim}_\text{\emph{B}} C\}
\]
for $\mathcal{L}^{dN}$-almost all $\mathbf{t} = (t_1, \dots, t_N) \in \times_{i=1}^N \mathbb{R}^d$.
\end{cor}

\begin{proof}
This follows immediately from Theorem \ref{main} since, for $\mathcal{L}^{dN}$-almost all $\mathbf{t} = (t_1, \dots, t_N) \in \times_{i=1}^N \mathbb{R}^d$, we have that $\overline{\dim}_\text{B} F_{\mathbf{t},\emptyset}$ is equal to the solution of
\[
\sum_{i=1}^N \text{Lip}(T_i)^s = 1
\]
which is also the similarity dimension of $F_{\mathbf{t},\emptyset}$.  This is a special case of a classical result of Falconer and Solomyak on the almost sure dimensions of self-affine sets, see \cite{affine, solomyak}.
\end{proof}

We conclude this section with an open question:
\begin{ques}
Is it true that
\[
\overline{\dim}_\text{\emph{B}} F_C = \max \{ \overline{\dim}_\text{\emph{B}} F_\emptyset, \ \overline{\dim}_\text{\emph{B}} C\}
\]
even if $\overline{\dim}_\text{\emph{B}} F_\emptyset<s$?  In particular, such systems cannot satisfy the SOSC.
\end{ques}

\subsection{Lower box dimension} \label{lowersect}

In this section we examine the lower box dimension.  Theorem \ref{main} gives us the following immediate Corollary which gives (basically trivial) bounds on the lower box dimension.
\begin{cor} \label{trivial bounds}
We have
\[
\max \{ \underline{\dim}_\text{\emph{B}} F_\emptyset , \ \underline{\dim}_\text{\emph{B}} C \} \ \leq \ \underline{\dim}_\text{\emph{B}} F_C  \ \leq  \ \overline{\dim}_\text{\emph{B}} F_C  \   \leq  \   \max \{ s, \ \overline{\dim}_\text{\emph{B}} C \}
\]
and if $\mathbb{I}$ satisfies the SOSC, then
\[
\max \{\underline{\dim}_\text{\emph{B}} F_\emptyset  , \ \underline{\dim}_\text{\emph{B}} C \} \ \leq \ \underline{\dim}_\text{\emph{B}} F_C  \ \leq  \ \overline{\dim}_\text{\emph{B}} F_C  \   \leq  \   \max \{ \underline{\dim}_\text{\emph{B}} F_\emptyset , \ \overline{\dim}_\text{\emph{B}} C \}.
\]
So we can compute the lower box dimension in three easy cases:
\begin{itemize}
\item[(1)] If the box dimension of $C$ exists and $\dim_\text{\emph{B}} C\geq s$, then
\[
\underline{\dim}_\text{\emph{B}} F_C = \overline{\dim}_\text{\emph{B}} F_C =\dim_\text{\emph{B}} C;
\]
\item[(2)] If the box dimension of $C$ exists and $\mathbb{I}$ satisfies the SOSC, then
\[
\underline{\dim}_\text{\emph{B}} F_C = \overline{\dim}_\text{\emph{B}} F_C = \max \{ \dim_\text{\emph{B}} F_\emptyset, \ \dim_\text{\emph{B}} C \};
\]
\item[(3)] If $\mathbb{I}$ satisfies the SOSC and $\overline{\dim}_\text{\emph{B}} C \leq s$, then
\[
\underline{\dim}_\text{\emph{B}} F_C = \overline{\dim}_\text{\emph{B}} F_C =s = \dim_\text{\emph{B}} F_\emptyset.
\]
\end{itemize}
Note that in each of the above cases the answer to Question \ref{quest2} is yes, i.e., 
\[
\underline{\dim}_\text{\emph{B}} F_C = \max \{ \underline{\dim}_\text{\emph{B}}  F_\emptyset, \ \underline{\dim}_\text{\emph{B}} C\}.
\]
\end{cor}
Even when $\mathbb{I}$ satisfies the SOSC, computing $\underline{\dim}_\text{B} F_C$ appears to be a subtle and difficult problem if $\max \{s, \, \underline{\dim}_\text{B} C \}<\overline{\dim}_\text{B} C$.  We will now briefly outline the reason for this.  Firstly, note that since lower box dimension is stable under taking closures, it follows from (\ref{structure}) that
\[
\underline{\dim}_\text{B} F_C = \underline{\dim}_\text{B} \overline{\mathcal{O}} =  \underline{\dim}_\text{B} \mathcal{O}.
\]
We can thus restrict our attention to the orbital set.  However, computing the dimension of $\mathcal{O}$ is difficult as it consists of copies of $C$ scaled by different amounts.  If the box dimension of $C$ does not exist, then the growth of the function $N_\delta(C)$ can vary wildly as $\delta \to 0$.  It turns out that the lower box dimension of $\mathcal{O}$ depends not only on  $\underline{\dim}_\text{B}$, $\overline{\dim}_\text{B}$ and $s$ but also on the behaviour of the function $\delta \mapsto N_\delta(C)$.  In order to analyse the behaviour of $N_\delta(C)$ we introduce a quantity which we call the \emph{covering regularity exponent} (CRE).  For $t \geq 0$ and $\delta \in (0,1]$, the ($t,\delta$)-CRE of $C$ is defined as
\begin{equation} \label{CREdef1}
p_{t,\delta}(C) = \sup \big\{p \in [0,1] : N_{\delta^p}(C) \geq \delta^{-pt} \big\}
\end{equation}
and the $t$-CRE is
\[
p_t(C) = \liminf_{\delta \to 0} \  p_{t,\delta}(C) .
\]
Roughly speaking, $p_{t,\delta}(C)$ tells you at scale $\delta$ how much you have to `scale up' to find a scale $\delta_0 \geq \delta$ where you need at least $\delta_0^{-t}$ sets to cover $C$, i.e., how far back you have to go to find a scale where the set is `hard' to cover.  In fact, the smaller $p_{t,\delta}(C)$ is, the further you have to go back. The constant $p_{t}(C)$ tells you  the `furthest away' you ever are from a scale where your set is `hard to cover', as you let $\delta$ tend to zero.  The following Lemma gives some simple but useful properties of the CREs.   First, recall that a metric space $(X,d)$ is \emph{Ahlfors regular} if $\dim_\H X<\infty$ and there exists a constant $\lambda>0$ such that, writing $\mathcal{H}^{\dim_\text{\H} X}$ to denote the Hausdorff measure in the critical dimension,
\[
\tfrac{1}{\lambda}\,  r^{\dim_\text{\H} X} \ \leq \  \mathcal{H}^{\dim_\text{\H} X} \big( B(x,r)\big) \  \leq \ \lambda \, r^{\dim_\text{\H} X} 
\]
for all $x \in X$ and all $0<r \leq \text{diam}(F)$.

\begin{lma} \label{keylem} \hspace{1mm}
\begin{itemize} 
\item[(1)] For all $t, \delta> 0$, we have $p_{t,\delta}(C), p_t(C) \in [0,1]$;
\item[(2)] $p_t(C)$ is decreasing in $t$ and if $t<\underline{\dim}_\text{\emph{B}}C$, then $p_t(C)=1$ and if $t>\overline{\dim}_\text{\emph{B}}C$, then $p_t(C)=0$;
\item[(3)] For all $\delta>0$ we have 
\[
N_{\delta^{p_{t,\delta}(C)}}(C) \geq \delta^{-p_{t,\delta}(C)t},
\]
i.e., the supremum in (\ref{CREdef1}) is obtained;
\item[(4)] For all $t > \underline{\dim}_\text{\emph{B}}C$, we have
\[
p_t(C) \ \leq \ \frac{\underline{\dim}_\text{\emph{B}}C}{t} \ < \  1;
\]
\item[(5)] For $\underline{\dim}_\text{\emph{B}}C< s< t < \overline{\dim}_\text{\emph{B}}C$ we have
\[
p_t(C)\  \leq \ \frac{s}{t} \, p_s(C);
\]
\item[(6)] Suppose $X$ is Ahlfors regular.  For all $t \in (\underline{\dim}_\text{\emph{B}}C, \, \overline{\dim}_\text{\emph{B}}C)$, we have
\[
p_t(C) \ \leq \ \frac{\underline{\dim}_\text{\emph{B}}C}{t}  \frac{\overline{\dim}_\text{\emph{B}}X-t}{\overline{\dim}_\text{\emph{B}}X-\underline{\dim}_\text{\emph{B}}C}.
\]
\end{itemize}
\end{lma}

\begin{figure}[H]
	\centering
	\includegraphics[width=160mm]{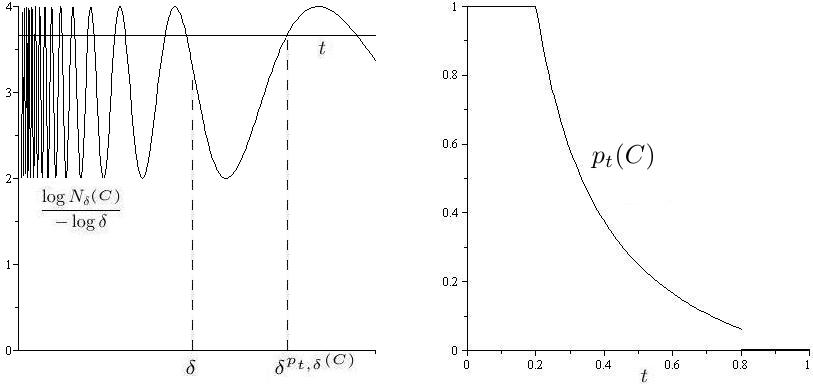}
	\caption{Left: A plot of $\log N_\delta(C)/(-\log\delta)$ for a set $C$ with distinct upper and lower box dimension.  A horizontal line is included at a value $t$ between the upper and lower box dimensions.  At the indicated point, $\delta$, we have that $N_{\delta}(C) <\delta^{-t}$ and so we have to `scale up' to $\delta_0 = \delta^{p_{t,\delta}(C)}$ to find a scale where $N_{\delta_0}(C) \geq\delta_0^{-t}$.
Right: A typical graph of $p_t(C)$ for a set $C$ with lower box dimension 0.2 and upper box dimension 0.8.
}
\end{figure}

We will prove Lemma \ref{keylem} in Section \ref{keylemproof}.  We will now use the CREs to obtain non-trivial bounds on the lower box dimension of $F_C$.  From now on we will assume that we are in the difficult case: $\max \{s, \, \underline{\dim}_\text{B} C \}<\overline{\dim}_\text{B} C$. 
\\ \\
The following theorem gives a \emph{lower bound} on the lower box dimension of $F_C$ and gives some sufficient conditions for the answer to Question \ref{quest2} to be \emph{no}.

\begin{thm} \label{lowerbound}
Suppose $(X,d)$ is Ahlfors regular and that $\mathbb{I}$ together with $C$ satisfies the COSC.  For all $t\geq0$ we have
\[
\underline{\dim}_\text{\emph{B}} F_C \ \geq \ p_{t}(C)\,  t \ + \ (1-p_{t}(C) ) \, s.
\]
In particular, if for some $t > \max\{s, \, \underline{\dim}_\text{\emph{B}}C\}$ we have
\[
p_t(C) > \max\Big\{0, \ \frac{\underline{\dim}_\text{\emph{B}}C-s}{t-s}\Big\},
\]
then
\[
\underline{\dim}_\text{\emph{B}} F_C > \max\{ \underline{\dim}_\text{\emph{B}} F_\emptyset, \,  \underline{\dim}_\text{\emph{B}} C\}.
\]
\end{thm}

We will prove Theorem \ref{lowerbound} in Section \ref{lowerproof}.  The next theorem gives an \emph{upper bound} on the lower box dimension of $F_C$ and gives some sufficient conditions for the answer to Question \ref{quest2} to be \emph{yes}.

\begin{thm} \label{upperbound}
For all $t> \max \{ s, \, \underline{\dim}_\text{\emph{B}} C\}$ we have
\[
\underline{\dim}_\text{\emph{B}} F_C \ \leq \ \max\{ t, \, s+p_{t}(C) \, t\}
\]
and, in particular, if $p_t(C) =0$ for $t > \max \{ s, \,  \underline{\dim}_\text{\emph{B}} C\}$, then
\[
\underline{\dim}_\text{\emph{B}} F_C \ \leq \  \max\{ s, \, \underline{\dim}_\text{\emph{B}} C\}
\]
and if, furthermore, the SOSC is satisfied, then
\[
\underline{\dim}_\text{\emph{B}} F_C \ = \ \max\{ \underline{\dim}_\text{\emph{B}} F_\emptyset, \,  \underline{\dim}_\text{\emph{B}} C\}.
\]
\end{thm}

We will prove Theorem \ref{upperbound} in Section \ref{upperproof}. We obtain the following (perhaps surprising) corollary in a very special case.
\begin{cor} \label{bzerocor}
If $\underline{\dim}_\text{\emph{B}} C = 0$ and $\mathbb{I}$ satisfies the SOSC, then
\[
\underline{\dim}_\text{\emph{B}} F_C \ = \ \max\{ \underline{\dim}_\text{\emph{B}} F_\emptyset, \,  \underline{\dim}_\text{\emph{B}} C\} = \underline{\dim}_\text{\emph{B}} F_\emptyset = s.
\]
\end{cor}
\begin{proof}
This follows immediately from Theorem \ref{upperbound} since Lemma \ref{keylem} (4) gives that $p_t(C) = 0$ for all $t>0$.
\end{proof}

The following Proposition proves the existence of compact sets with the extremal behaviour described in Theorems \ref{lowerbound}--\ref{upperbound}.  In particular, Proposition \ref{exlem} (2) combined with Theorem \ref{lowerbound} gives a negative answer to Question \ref{quest2}.

\begin{prop} \label{exlem}  Let $X = [0,1]^d$ for some $d \in \mathbb{N}$.
\begin{itemize} 
\item[(1)]  For all $0<b<t<B\leq d$, there exists a compact set $C \subseteq X$ such that $\underline{\dim}_\text{\emph{B}}C = b< B = \overline{\dim}_\text{\emph{B}}C$ and $p_t(C) = 0$ for all $t\geq b$;
\item[(2)] For all $0<b<B\leq d$, there exists a compact set $C\subseteq X$ such that $\underline{\dim}_\text{\emph{B}}C = b< B = \overline{\dim}_\text{\emph{B}}C$ and
\[
p_t(C)=\frac{b}{t} \ \frac{d-t}{d-b}.
\]
for all $t \in (b,B)$.  In particular, such a $C$ shows that the upper bound in Lemma \ref{keylem} (6) is sharp.
\end{itemize}
\end{prop}

We will prove Proposition \ref{exlem} in Section \ref{exlemproof}.  Although we specialise to the case where $X$ is the unit cube, the result applies in much more general situations.  However, as we only require them to provide examples, we omit any further technical details.
\\ \\
The case where the condensation set is constructed as in Proposition \ref{exlem} (2) is an interesting case.  Not only does it provide a negative answer to Question \ref{quest2} but we also obtain an explicit (non-trivial) formula for $p_t(C)$.   We obtain the following corollary in this situation.

\begin{cor}
Let $X = [0,1]^d$, let $\mathbb{I}=\{ S_{1}, \dots,  S_{N}\}$ be an IFS of similarities on $X$ and fix a non-empty compact set $C\subset[0,1]^d$ such that
\[
p_t(C)=\frac{\underline{\dim}_\text{\emph{B}}C}{t} \frac{d-t}{d-\underline{\dim}_\text{\emph{B}}C}
\]
for all $t \in (\underline{\dim}_\text{\emph{B}}C , \,  \overline{\dim}_\text{\emph{B}}C)$.  Furthermore assume that $\mathbb{I}$ together with $C$ satisfies the COSC.  Then
\[
\frac{\underline{\dim}_\text{\emph{B}}C}{t}  \frac{d-t}{d-\underline{\dim}_\text{\emph{B}}C} \ ( t \,  -  \,  s ) \, + \,  s \ \leq \ \underline{\dim}_\text{\emph{B}} F_C \ \leq \ \max\Big\{ t, \ s \, + \, \underline{\dim}_\text{\emph{B}}C  \frac{d-t}{d-\underline{\dim}_\text{\emph{B}}C}\Big\}
\]
for all $t \in (\underline{\dim}_\text{\emph{B}}C , \, \overline{\dim}_\text{\emph{B}}C)$.
\end{cor}

Write $L(t)$ and $U(t)$ for the lower and upper bounds for $\underline{\dim}_\text{B} F_C$ given in the above Corollary.  We will now provide a plot of these as functions of $t$ in two typical situations.  Of course the best lower and upper bounds for $\underline{\dim}_\text{B} F_C$ are really the supremum and infimum of $L(t)$ and $U(t)$ respectively. In both cases we let $X=[0,1]^5$.  For the plot on the left, we let $\underline{\dim}_\text{B} C=1$, $s=1.5$ and $\overline{\dim}_\text{B} C=4.5$.  For the plot on the right, we let $\underline{\dim}_\text{B} C=s=1$ and $\overline{\dim}_\text{B} C=2$.  In the first case the trivial bounds bounds from Corollary \ref{trivial bounds} have been improved from $[1.5,  4.5]$ to $[1.756,  2.2]$ and in the second case the trivial bounds have been improved from $[1,  2]$ to $[1.375,   1.8]$.

\begin{figure}[H]
	\centering
	\includegraphics[width=155mm]{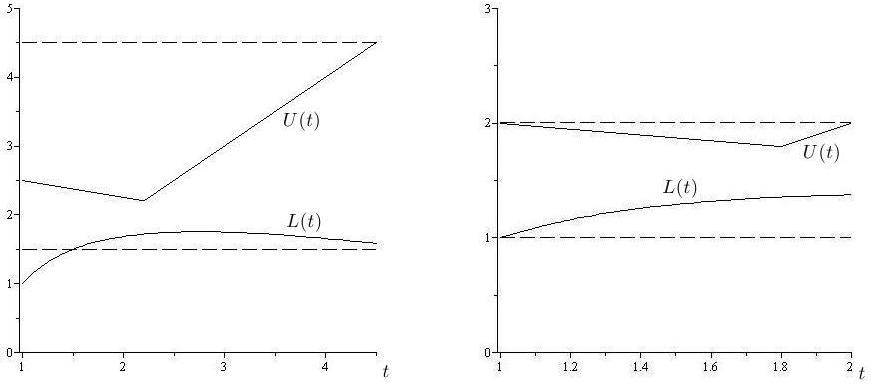}
	\caption{Two graphs showing the upper and lower bounds on the lower box dimension of $F_C$.  $U(t)$ and $L(t)$ are plotted as solid lines and the trivial bounds from Corollary \ref{trivial bounds} are plotted as dashed lines.  We can clearly see a significant improvement on the trivial bounds and in both cases $\underline{\dim}_\text{B} F_C > \max\{ \underline{\dim}_\text{B} F_\emptyset, \,  \underline{\dim}_\text{B} C \}$. 
}
\end{figure}

We will present one final corollary which summarises the `bad behaviour' of the lower box dimension of inhomogeneous self-similar sets.

\begin{cor}
Regardless of separation conditions, the lower box dimension of $F_C$ is not in general given by a function of the numbers:
\[
\underline{\dim}_\text{\emph{B}} C, \ \overline{\dim}_\text{\emph{B}} C, \ \dim_\text{\emph{H}} C, \ \dim_\text{\emph{P}} C, \ \dim_\text{\emph{B}} F_\emptyset \text{  and  } \ s.
\]
This is in stark contrast to the situation for the countably stable dimensions and the upper box dimension.
\end{cor}

\begin{proof}
This follows from the results in this section.
\end{proof}

\section{Proofs}

\subsection{Preliminary results and notation}

Fix an IFS $\mathbb{I}=\{ S_{1}, \dots,  S_{N}\}$ where each $S_i$ is a similarity and fix a compact condensation set, $C\subseteq X$.  Write $\mathcal{I} = \{1, \dots, N\}$, $L_{\min} = \min_{i \in \mathcal{I}} \text{Lip}(S_i)$ and $L_{\max} = \max_{i \in \mathcal{I}} \text{Lip}(S_i)$.  Let
\[
\mathcal{I}^* = \bigcup_{k \in \mathbb{N} } \mathcal{I}^k
\]
denote the set of all finite words over $\mathcal{I}$.  For $\textbf{i}  = (i_1, \dots, i_k) \in \mathcal{I}^*$, write $S_\textbf{i} = S_{i_1} \circ \cdots \circ S_{i_k}$, let $\textbf{i}_-  = (i_1, \dots, i_{k-1})$ and write $\lvert \textbf{i} \rvert = k$ to denote the length of the string $\textbf{i}$.  For $\delta \in (0,1]$, define a $\delta$-stopping, $\mathcal{I}(\delta)$, by
\[
\mathcal{I}(\delta) = \big\{ \textbf{i} \in \mathcal{I}^* : \text{Lip}(S_\textbf{i}) < \delta \leq \text{Lip}(S_{ \textbf{i}_-}) \big\},
\]
where we assume for convenience that $ \text{Lip}(S_{\omega}) = 1$, where $\omega$ is the empty word.
\begin{lma} \label{stopsize}
For all $\delta \in (0,1]$, we have
\[
\delta^{-s} \ \leq \  \lvert  \mathcal{I}(\delta) \rvert  \ \leq \ L_{\min}^{-s} \, \delta^{-s}.
\]
\end{lma}

\begin{proof}
Repeated application of Hutchinson's formula (\ref{hutch}) gives
\[
\sum_{\textbf{i} \in \mathcal{I}(\delta)} \text{Lip}(S_\textbf{i})^{s}=1
\]
from which we deduce
\begin{equation} \label{ub}
1 = \sum_{\textbf{i} \in \mathcal{I}(\delta)} \text{Lip}(S_\textbf{i})^{s} \geq \sum_{\textbf{i} \in \mathcal{I}(\delta)} (\delta \, L_{\min})^s = \lvert \mathcal{I}(\delta) \rvert \,  (\delta \, L_{\min})^s 
\end{equation}
and
\begin{equation} \label{lb}
1 = \sum_{\textbf{i} \in \mathcal{I}(\delta)} \text{Lip}(S_\textbf{i})^{s} \leq \sum_{\textbf{i} \in \mathcal{I}(\delta)} \delta^s = \lvert \mathcal{I}(\delta) \rvert \,  \delta^s.
\end{equation}
The desired upper and lower bounds now follow from (\ref{ub}) and (\ref{lb}) respectively.
\end{proof}

\begin{lma} \label{sumsize}
For all $t > s$ we have
\[
\sum_{\textbf{\emph{i}} \in \mathcal{I}^*} \text{\emph{Lip}}(S_\textbf{\emph{i}})^{t} \, = \, b_t \, < \, \infty
\]
for some constant $b_t$ depending only on $t$.
\end{lma}

\begin{proof}
This is a standard fact but we include the simple proof for completeness and to define the constant $b_t$.  Since $t > s$ we have $\sum_{i  \in \mathcal{I}} \text{Lip}(S_i)^{t}<1$.  It follows that
\[
\sum_{\textbf{i} \in \mathcal{I}^*} \text{Lip}(S_\textbf{i})^{t} = \sum_{k=1}^\infty \sum_{\textbf{i} \in \mathcal{I}^k} \text{Lip}(S_\textbf{i})^{t} = \sum_{k=1}^\infty \Bigg( \sum_{i \in \mathcal{I}} \text{Lip}(S_i)^{t} \Bigg)^k < \infty,
\]
which proves the Lemma, setting $b_t = \sum_{k=1}^\infty \Big( \sum_{i \in \mathcal{I}} \text{Lip}(S_i)^{t} \Big)^k$.
\end{proof}

\begin{lma} \label{setsize}
For all $\delta \in (0,1)$, we have
\[
\lvert \{\textbf{\emph{i}} \in \mathcal{I}^*: \delta \leq \text{\emph{Lip}}(S_\textbf{\emph{i}})   \} \rvert \ \leq \ \frac{\log \delta}{\log L_{\max}}  \  \delta^{-s}.
\]
\end{lma}

\begin{proof}
Let $\delta \in (0,1)$ and suppose $\textbf{i} \in \mathcal{I}^*$ is such that $\delta \leq \text{Lip}(S_\textbf{i})$.  It follows that $\delta \leq L_{\max}^{\lvert \textbf{i} \rvert}$ and hence
\begin{equation} \label{length}
\lvert \textbf{i} \rvert \leq \frac{\log \delta}{\log L_{\max}}.
\end{equation}
Repeatedly applying Hutchison's formula (\ref{hutch}) gives
\begin{eqnarray*}
 \tfrac{\log \delta}{\log L_{\max}}   \ \geq \  \sum_{\substack{l \in \mathbb{N}:\\ \\
l \leq  \frac{\log \delta}{\log L_{\max}}}}1 \ &\geq& \  \sum_{\substack{l \in \mathbb{N}:\\ \\
l \leq  \frac{\log \delta}{\log L_{\max}}}} \sum_{\textbf{i} \in \mathcal{I}^l} \text{Lip}(S_\textbf{i})^s \\ \\
&\geq& \ \sum_{\substack{l \in \mathbb{N}:\\ \\
l \leq  \frac{\log \delta}{\log L_{\max}}}} \  \sum_{\substack{\textbf{i} \in \mathcal{I}^l: \\ \\
\delta \leq \text{Lip}(S_\textbf{i})}} \text{Lip}(S_\textbf{i})^s 
\\ \\
&\geq& \ \sum_{\substack{l \in \mathbb{N}:\\ \\
l \leq  \frac{\log \delta}{\log L_{\max}}}} \  \sum_{\substack{\textbf{i} \in \mathcal{I}^l: \\ \\
\delta \leq \text{Lip}(S_\textbf{i})}} \delta^s \\ \\
&=& \lvert \{\textbf{i} \in \mathcal{I}^*: \delta \leq \text{Lip}(S_\textbf{i})   \} \rvert \, \delta^s
\end{eqnarray*}
by (\ref{length}), which proves the result.
\end{proof}

\subsection{Proof of Lemma \ref{keylem}}  \label{keylemproof}

\emph{Proof of (1):} This follows immediately from the definition of $p_{t,\delta}(C)$ and the fact that the set
\[
\big\{p \in [0,1] : N_{\delta^p}(C) \geq \delta^{-pt} \big\}
\]
is never empty as it always contains the point 0.
\\ \\
\emph{Proof of (2):}  It is clear that $p_t(C)$ is decreasing in $t$.  If $t< \underline{\dim}_\text{B} C$, then there exists $\delta_0 \in (0,1]$ such that for all $\delta<\delta_0$ we have
\[
N_{\delta}(C) \geq \delta^{-t}
\]
which implies that if $\delta<\delta_0$, then  $p_{t,\delta}(C)=1$, which completes the proof.  The proof that if $t> \overline{\dim}_\text{B} C$, then $p_{t}(C)=0$ is similar and omitted.
\\ \\
\emph{Proof of (3):} Let $t >0$ and $\delta \in (0,1]$ and without loss of generality assume that $p_{t,\delta}(C)>0$.  By the definition of $p_{t,\delta}(C)$ we may choose arbitrarily small $\varepsilon \in \big(0,p_{t,\delta}(C)\big)$ such that
\begin{equation} \label{choosingepsilon}
N_{\delta^{p_{t,\delta}(C)-\varepsilon}}(C) \geq \delta^{-(p_{t,\delta}(C)-\varepsilon)t}.
\end{equation}
It follows from this that
\[
N_{\delta^{p_{t,\delta}(C)}}(C) \geq N_{\delta^{p_{t,\delta}(C)-\varepsilon}}(C) \geq \delta^{-(p_{t,\delta}(C)-\varepsilon)t} = \delta^{-p_{t,\delta}(C)t} \, \delta^{\varepsilon t}
\]
and letting $\varepsilon \to 0$ through values satisfying (\ref{choosingepsilon}) proves the result.
\\ \\
\emph{Proof of (4):} Let $t>\underline{\dim}_\text{B}C$ and $\varepsilon \in (0, t- \underline{\dim}_\text{B}C)$.  By the definition of lower box dimension, there exists arbitrarily small $\delta>0$ such that
\[
N_\delta(C) \leq \delta^{-(\underline{\dim}_{\text{B}}C+\varepsilon)}.
\]
Fix such a $\delta \in (0,1)$ and since $N_\delta(C)$ increases as $\delta$ decreases,
\[
\delta^{-p_{t,\delta}(C)t} \leq N_{\delta^{p_{t,\delta}(C)}}(C) \leq N_\delta(C) \leq \delta^{-(\underline{\dim}_{\text{B}}C+\varepsilon)}.
\]
Taking logs and dividing by $-t\log \delta$ yields
\[
p_{t,\delta}(C) \leq \frac{\underline{\dim}_{\text{B}}C+\varepsilon}{t}
\]
and since we can find arbitrarily small $\delta$ satisfying the above inequality, the desired upper bound follows.
\\ \\
\emph{Proof of (5):} Let $\underline{\dim}_\text{B}C< s< t< \overline{\dim}_\text{B}C$.  It follows from Lemma \ref{keylem} (4) above that $ p_s(C)<1$ and so we may choose $\varepsilon \in (0,1-p_s(C)]$.  It follows that there exists $\delta \in (0, \varepsilon)$ such that $p_{s,\delta}(C) < p_s(C)+\varepsilon \leq 1$.  This implies that
\[
N_{\delta^{p_s(C)+\varepsilon}}(C)< \delta^{-(p_s(C)+\varepsilon)s}.
\]
Using this, Lemma \ref{keylem} (3), and the fact that $N_\delta(C)$ increases as $\delta$ decreases, we have
\[
\delta^{-p_{t,\delta}(C)t} \leq N_{\delta^{p_{t,\delta}(C)}}(C) \leq  N_{\delta^{p_s(C)+\varepsilon}}(C)< \delta^{-(p_s(C)+\varepsilon)s}.
\]
Taking logs and dividing by $-t\log \delta$ yields
\[
p_{t,\delta}(C) \ \leq \  \frac{s}{t} \, (p_s(C)+\varepsilon)
\]
and since we can find arbitrarily small $\delta$ satisfying the above inequality, the desired upper bound follows.
\\ \\
\emph{Proof of (6):} Let $t \in (\underline{\dim}_\text{B}C, \, \overline{\dim}_\text{B}C)$ and $\varepsilon \in (0, t- \underline{\dim}_\text{B}C)$.  Following the argument used in the proof of Lemma \ref{keylem} (4), we can find arbitrarily small $\delta \in (0,1)$ such that
\begin{equation} \label{Nestimate}
N_\delta(C) \leq \delta^{-(\underline{\dim}_\text{B}C+\varepsilon)}  \qquad \text{and} \qquad p_{t,\delta}(C) \leq \frac{\underline{\dim}_{\text{B}}C+\varepsilon}{t} \leq 1.
\end{equation}
Fix such a $\delta$.  Since $X$ is Ahlfors regular, it follows that there exists constants $K\geq 1$ and $\rho \in (0,1]$ such that any ball of radius $\delta<\rho$ can be covered by fewer than
\begin{equation} \label{ahl}
K \, \Big(\frac{\delta}{\delta_0} \Big)^{\overline{\dim}_\text{B}X}
\end{equation}
 balls of radius $\delta_0\leq\delta<\rho$.  Let
\begin{equation} \label{mbound}
m=\max \Bigg\{1, \  \frac{\log K}{(\overline{\dim}_\text{B}X-t) \log \delta}+\frac{\overline{\dim}_\text{B}X-\underline{\dim}_\text{B}C-\varepsilon}{\overline{\dim}_\text{B}X-t}  \Bigg\}.
\end{equation}
Let $\delta' = \delta^{q} \in (\delta^m, \delta)$ for some $q \in (1,m)$.  A simple calculation combining (\ref{Nestimate}, \ref{ahl}, \ref{mbound}) yields that
\[
N_{\delta'}(C) \ = \ N_{\delta^q}(C) \ \leq \  K  \Big(\frac{\delta}{\delta^q} \Big)^{\overline{\dim}_\text{B}X} N_{\delta}(C) \ \leq \ K  \Big(\frac{\delta}{\delta^q} \Big)^{\overline{\dim}_\text{B}X} \delta^{-(\underline{\dim}_{\text{B}}C+\varepsilon)} \ < \ \delta^{-qt} = (\delta')^{-t}.
\]
Note that if $m=1$, then this is vacuously true, but indeed $m>1$ for sufficiently small $\varepsilon$ and $\delta$.  It follows that
\[
N_{\delta'}(C) < (\delta')^{-t}
\]
for all $\delta' \in  (\delta^m, \delta) \cup  [\delta, \delta^{p_{t,\delta}(C)}) = (\delta^m, \delta^{p_{t,\delta}(C)})$.  This, combined with the fact that
\[
N_{(\delta^m)^{p_{t,\delta}(C)/m}}(C) \ =  \ N_{\delta^{p_{t,\delta}(C)}}(C) \ \geq \  \delta^{-p_{t,\delta}(C) t}  \ = \  (\delta^{m})^{-(p_{t,\delta}(C)/m) t}
\]
by the definition of $p_{t,\delta}(C)$, yields that $p_{t,\delta^m}(C)  =   p_{t,\delta}(C)/m$.  Hence
\[
p_{t,\delta^m}(C) \ = \  \frac{p_{t,\delta}(C)}{m} \ \leq  \  \frac{\underline{\dim}_{\text{B}}C+\varepsilon}{t} \ \bigg(\frac{\log K}{(\overline{\dim}_\text{B}X-t) \log \delta}+\frac{\overline{\dim}_\text{B}X-\underline{\dim}_\text{B}C-\varepsilon}{\overline{\dim}_\text{B}X-t} \bigg)^{-1}
\]
by (\ref{Nestimate}, \ref{mbound}).  Letting $\delta \to 0$ through values satisfying (\ref{Nestimate}) yields
\[
p_{t}(C) \ \leq \  \frac{\underline{\dim}_{\text{B}}C+\varepsilon}{t} \ \frac{\overline{\dim}_\text{B}X-t}{\overline{\dim}_\text{B}X-\underline{\dim}_\text{B}C-\varepsilon}
\]
and finally letting $\varepsilon \to 0$ we have
\[
p_{t}(C) \ \leq \  \frac{\underline{\dim}_{\text{B}}C}{t} \ \frac{\overline{\dim}_\text{B}X-t}{\overline{\dim}_\text{B}X-\underline{\dim}_\text{B}C}
\]
as required. \hfill \qed

\subsection{Proof of Theorem \ref{main}}  \label{mainproof}

By monotonicity of upper box dimension, we have $\max \{ \overline{\dim}_\text{B} F_\emptyset, \ \overline{\dim}_\text{B} C\} \ \leq \ \overline{\dim}_\text{B} F_C$.  We will now prove the other inequality.  Since upper box dimension is finitely stable it suffices to show that
\[
\overline{\dim}_\text{B} \mathcal{O} \leq \max\{s, \ \overline{\dim}_\text{B} C\}.
\]
Let $t>\max\{s, \ \overline{\dim}_\text{B} C\}$.  It follows from the definition of upper box dimension that there exists a constant $c_t>0$ such that
\begin{equation} \label{boundy}
N_\delta(C) \leq c_t \, \delta^{-t}
\end{equation}
for all $\delta \in (0,1]$.  Also note that since $X$ is compact, the number of balls of radius 1 required to cover $X$ is a finite constant $N_1(X)$.  Let $\delta \in (0,1]$.  We have
\begin{eqnarray*}
N_\delta(\mathcal{O}) = N_\delta \Bigg(  C \cup \bigcup_{\textbf{i} \in \mathcal{I}^*} S_{\textbf{i}}(C) \Bigg)  &\leq& \sum_{\substack{\textbf{i} \in \mathcal{I}^*: \\ \\
\delta \leq \text{Lip}(S_\textbf{i})}} N_\delta \big( S_{\textbf{i}}(C) \big) \ +  \   N_\delta  \Bigg( \  \bigcup_{\substack{\textbf{i} \in \mathcal{I}^*: \\ \\
\delta> \text{Lip}(S_\textbf{i})}} S_{\textbf{i}}(C) \ \Bigg) \ + \ N_\delta(C)\\ \\
&\leq& \sum_{\substack{\textbf{i} \in \mathcal{I}^*: \\ \\
\delta \leq \text{Lip}(S_\textbf{i})}} N_{\delta / \text{Lip}(S_\textbf{i})} (C) \ +  \   N_\delta  \Bigg( \  \bigcup_{\textbf{i} \in \mathcal{I}(\delta)} S_{\textbf{i}}(X) \ \Bigg) \ + \ N_\delta(C)\\ \\
&\leq& \sum_{\substack{\textbf{i} \in \mathcal{I}^*: \\ \\
\delta \leq \text{Lip}(S_\textbf{i})}} c_t \, \big(\delta / \text{Lip}(S_\textbf{i})\big)^{-t} \ +  \   \sum_{\textbf{i} \in \mathcal{I}(\delta)} N_{\delta / \text{Lip}(S_\textbf{i})}(X) \ + \ c_t \, \delta^{-t} \qquad \text{by (\ref{boundy})}\\ \\
&\leq& c_t \, \delta^{-t}  \sum_{\substack{\textbf{i} \in \mathcal{I}^*: \\ \\
\delta \leq \text{Lip}(S_\textbf{i})}} \text{Lip}(S_\textbf{i})^{t} \ +  \   N_1(X) \, \lvert  \mathcal{I}(\delta) \rvert \ + \ c_t \, \delta^{-t}  \\ \\
&\leq& c_t \, \delta^{-t}  \, \sum_{\textbf{i} \in \mathcal{I}^*} \text{Lip}(S_\textbf{i})^{t} \ +  \   N_1(X) \, L_{\min}^{-s} \, \delta^{-s} \ + \ c_t \, \delta^{-t} \qquad \quad  \text{by Lemma \ref{stopsize}}\\ \\
&\leq& \big(c_t \, b_t \, +  \,   N_1(X) \, L_{\min}^{-s} \, + \, c_t \big) \, \delta^{-t}
\end{eqnarray*}
by Lemma \ref{sumsize}, from which it follows that $\overline{\dim}_\text{B} F_C =\overline{\dim}_\text{B} \mathcal{O}  \leq t$ and since $t$ can be chosen arbitrarily close to $\max\{s, \ \overline{\dim}_\text{B} C\}$, we have proved the Theorem. \hfill \qed

\subsection{Proof of Theorem \ref{lowerbound}}  \label{lowerproof}

Suppose $(X,d)$ is Ahlfors regular and that $\mathbb{I}$, together with $C$, satisfies the COSC.  We begin with two simple technical lemmas.
\begin{lma} \label{openlem}
Let $a,b >0$, let $\{U_i\}$ be a collection of disjoint open subsets of $X$ and suppose that each $U_i$ contains a ball of radius $ar$ and is contained in a ball of radius $br$.  Then any ball of radius $r$ intersects no more than
\[
\lambda^2 \, \Big(\frac{1+2b}{a}\Big)^{\dim_\text{\emph{H}} X}
\]
of the closures $\{\overline{U}_i\}$.
\end{lma}

This is a trivial modification of a standard result in Euclidean space, see \cite[Lemma 9.2]{falconer}, but for completeness we include the simple proof.
\begin{proof}
For each $i$ let $B_i$ denote the ball of radius $ar$ contained in $U_i$ and note that these balls are pairwise disjoint.  Fix $x \in X$ and suppose $B(x,r) \cap \overline{U}_i \neq \emptyset$ for some $i$.  It follows that $\overline{U}_i  \subseteq B\big(x, (1+2b)r\big)$.  Suppose the number of $i$ such that $B(x,r) \cap \overline{U}_i \neq \emptyset$ is equal to $N$.  Then
\[
N \, \tfrac{1}{\lambda}\,  (ar)^{\dim_\text{\H} X}  \ \leq \sum_{i: B(x,r) \cap \overline{U}_i \neq \emptyset} \mathcal{H}^{\dim_\text{\H} X} \big( B_i\big) \  \leq \  \mathcal{H}^{\dim_\text{\H} X}\Big( B\big(x, (1+2b)r\big) \Big)\  \leq\ \lambda \, \big((1+2b)r\big)^{\dim_\text{\H} X} 
\]
and solving for $N$ proves the lemma.
\end{proof}

\begin{lma} \label{opendisjoint}
Let $\delta \in (0,1]$ and $\textbf{\emph{i}}, \textbf{\emph{j}} \in \mathcal{I}(\delta)$ with $\textbf{\emph{i}} \neq \textbf{\emph{j}}$.  Writing $U$ for the open set used in the COSC, we have
\[
S_{\textbf{\emph{i}}}(U) \cap S_{\textbf{\emph{j}}}(U) = \emptyset.
\]
\end{lma}

\begin{proof}
This is a simple consequence of the COSC (in fact the OSC is enough) and the fact that neither $\textbf{i}$ nor $\textbf{j}$ is a subword of the other.
\end{proof}

We now turn to the proof of Theorem \ref{lowerbound}.

\begin{proof} If $0 \leq t \leq \max\{s, \underline{\dim}_\text{B}C\}$, then the result is clearly true (and not an improvement on Corollary \ref{trivial bounds}) so assume that $t > \max\{s, \underline{\dim}_\text{B}C\}$ and let $\varepsilon \in (0,1]$.  Choose $\delta_0 \in (0,1]$ such that for all $\delta \in (0, \delta_0]$ we have $p_{t,\delta}(C) \geq p_t(C) - \varepsilon$.  Fix $\delta \in (0,\delta_0]$ and finally, to simplify notation, write $p_{t,\delta} = p_{t,\delta}(C)$ and $p_t = p_t(C)$.  We will now consider two cases.
\\ \\
\emph{Case 1: Assume that $\delta^{1-p_{t,\delta}} \, L_{\min}^{-1}  \leq 1$.}
\\ \\
Let $U$ be the open set used for the COSC and choose $a,b>0$ such that $U$ contains a ball of radius $a$ and is contained in a ball of radius $b$.  It follows that for each $\textbf{i} \in \mathcal{I}(\delta^{1-p_{t,\delta} }\, L_{\min}^{-1})$ the image $S_{\textbf{i}}(U)$ is an open set which contains a ball of radius $a \, \delta^{1-p_{t,\delta}}$ and is contained in a ball of radius $b \, L_{\min}^{-1} \, \delta^{1-p_{t,\delta}}$.  Furthermore, it follows from Lemma \ref{opendisjoint} that the sets
\[
\big\{S_{\textbf{i}}(U) : \textbf{i} \in \mathcal{I}(\delta^{1-p_{t,\delta} }\, L_{\min}^{-1}) \big\}
\]
are pairwise disjoint.  Since, for each $\textbf{i} \in \mathcal{I}(\delta^{1-p_{t,\delta} }\, L_{\min}^{-1})$, we have $S_{\textbf{i}}(C) \subseteq \overline{S_{\textbf{i}}(U)}$, it follows from Lemma \ref{openlem} that any ball of radius $\delta^{1-p_{t,\delta}}$, and hence any set of diameter $\delta$, can intersect no more than
\[
\kappa:=\lambda^2 \, \Big(\frac{1+2bL_{\min}^{-1}}{a}\Big)^{\dim_\H X}
\]
of the sets
\[
\big\{S_{\textbf{i}}(C) : \textbf{i} \in \mathcal{I}(\delta^{1-p_{t,\delta} }\, L_{\min}^{-1}) \big\}.
\]
Whence
\begin{eqnarray*}
N_\delta(\mathcal{O}) = N_\delta \Bigg(C \cup   \bigcup_{\textbf{i} \in \mathcal{I}^*} S_{\textbf{i}}(C) \Bigg)  &\geq& \kappa^{-1}\sum_{\textbf{i} \in \mathcal{I}(\delta^{1-p_{t,\delta} }\, L_{\min}^{-1})} N_\delta \big(  S_{\textbf{i}}(C) \big) \\ \\
&=& \kappa^{-1}\sum_{\textbf{i} \in \mathcal{I}(\delta^{1-p_{t,\delta} }\, L_{\min}^{-1})} N_{\delta/\text{Lip}(S_{\textbf{i}})} (C) \\ \\
&\geq& \kappa^{-1}\sum_{\textbf{i} \in \mathcal{I}(\delta^{1-p_{t,\delta} }\, L_{\min}^{-1})} N_{\delta^{p_{t,\delta} }} (C) \\ \\
&\geq& \kappa^{-1} \, \delta^{-tp_{t,\delta} } \ \lvert \mathcal{I}(\delta^{1-p_{t,\delta} }\, L_{\min}^{-1}) \rvert \qquad \qquad  \text{by Lemma \ref{keylem} (3)}\\ \\
&\geq&\kappa^{-1} \, \delta^{-tp_{t,\delta} } \ (\delta^{1-p_{t,\delta} }\, L_{\min}^{-1})^{-s}  \qquad \qquad  \text{by Lemma \ref{stopsize}}\\ \\
&=&\kappa^{-1} \, L_{\min}^{s} \,  \delta^{-\big(p_{t,\delta} t +(1-p_{t,\delta} )s\big)}\\ \\
&\geq& \kappa^{-1} \, L_{\min}^{s} \, \delta^{-\big((p_{t}-\varepsilon) t +(1-(p_{t}-\varepsilon) )s\big)}
\end{eqnarray*}
from which it follows that $\underline{\dim}_\text{B} F_C = \underline{\dim}_\text{B} \mathcal{O} \geq (p_{t}-\varepsilon) t +(1-(p_{t}-\varepsilon) )s$.
\\ \\
\emph{Case 2: Assume that $\delta^{1-p_{t,\delta}} \, L_{\min}^{-1}  > 1$.}
\\ \\
Note that our assumption implies that $1 \geq \delta^{-(1-p_{t,\delta}) s} \, L_{\min}^s$.  It follows that
\[
N_\delta(\mathcal{O}) \ \geq \   N_{\delta^{p_{t,\delta}}} (C) \ \geq \ \delta^{-p_{t,\delta}t}\  \geq \  \delta^{-(1-p_{t,\delta}) s} \, L_{\min}^s \,  \delta^{-p_{t,\delta}t} \ \geq \ L_{\min}^{s} \, \delta^{-\big((p_{t}-\varepsilon) t +(1-(p_{t}-\varepsilon) )s\big)} 
\]
from which it follows that $\underline{\dim}_\text{B} \mathcal{O} \geq (p_{t}-\varepsilon) t +(1-(p_{t}-\varepsilon) )s$.
\\ \\
Combining Cases 1--2 and letting $\varepsilon$ tend to zero proves the Theorem.

\end{proof}

\subsection{Proof of Theorem \ref{upperbound}}   \label{upperproof}

We begin with a simple technical Lemma.

\begin{lma} \label{findd}
Let $t \geq 0$.  If $p_t(C) <1$, then for all $\varepsilon \in \big(0, 1-p_t(C) \big)$, there exists $\delta \in (0,\varepsilon)$ such that
\[
p_t(C)-\varepsilon<p_{t,\delta}(C)<p_t(C)+\varepsilon
\]
and, for all $\delta_0 \in [\delta, \delta^{p_t(C)}]$, we have
\[
N_{\delta_0}(C) \leq \delta_0^{-t}.
\]
\end{lma}

\begin{proof}
Since $p_t(C) <1$, it follows that for all $\varepsilon \in \big(0, 1-p_t(C) \big)$, there exists $\delta \in (0,\varepsilon)$ such that $p_t(C)- \varepsilon < p_{t,\delta}(C)<p_t(C)+\varepsilon<1$.  By the definition of $p_{t,\delta}(C)$ this implies that for all $\delta_0 \in  [\delta, \delta^{p_t(C)+\varepsilon}]$ we have
\[
N_{\delta_0}(C) \leq \delta_0^{-t}
\]
which completes the proof.
\end{proof}

We will now turn to the proof of Theorem  \ref{upperbound}.
\begin{proof}
Let  $t> \max\{ s , \, \underline{\dim}_\text{B} C\}$.  By Lemma \ref{keylem} (4), we have $p_{t}(C) \leq \underline{\dim}_\text{B}C/t<1$ and so by Lemma \ref{findd}, for all $\varepsilon \in \big(0, 1-p_t(C) \big)$, there exists $\delta \in (0,\varepsilon)$ such that
\begin{equation} \label{boundp1}
p_t(C)-\varepsilon<p_{t,\delta}(C)<p_t(C)+\varepsilon
\end{equation}
and for all $\delta_0 \in [\delta, \delta^{p_t(C)}]$ we have
\begin{equation} \label{boundp2}
N_{\delta_0}(C) \leq \delta_0^{-t}.
\end{equation}
Fix $\varepsilon \in (0, 1-p_t(C) )$ and choose $\delta \in (0,\varepsilon)$ satisfying (\ref{boundp1}, \ref{boundp2}).  Write $p_{t,\delta} = p_{t,\delta}(C)$ and $p_t = p_t(C)$.  We have
\begin{eqnarray*}
N_\delta(\mathcal{O}) &=& N_\delta \Big( C \cup \bigcup_{\textbf{i} \in \mathcal{I}^*} S_{\textbf{i}}(C) \Big) \\ \\
&\leq& \sum_{\substack{\textbf{i} \in \mathcal{I}^*: \\ \\ \delta^{1-p_{t,\delta}-\varepsilon}  \, \leq \,  \text{Lip}(S_i) \, < \, 1}} N_\delta \big(  S_{\textbf{i}}(C) \big) \ + \ \sum_{\substack{\textbf{i} \in \mathcal{I}^*: \\ \\ \delta \, \leq \,  \text{Lip}(S_i) \,  < \, \delta^{1-p_{t,\delta}-\varepsilon}}} N_\delta \big(  S_{\textbf{i}}(C) \big) \\ \\
&\quad& \qquad   \qquad  \ + \   N_\delta \ \Bigg( \  \bigcup_{\substack{\textbf{i} \in \mathcal{I}^*: \\ \\
 \text{Lip}(S_\textbf{i}) \, < \, \delta}} S_{\textbf{i}}(C) \ \Bigg) \ + \ N_\delta(C) \\ \\
&\leq& \sum_{\substack{\textbf{i} \in \mathcal{I}^*: \\ \\ \delta^{1-p_{t,\delta}-\varepsilon} \,  \leq \,  \text{Lip}(S_i)  \, < \, 1}} N_{\delta/\text{Lip}(S_{\textbf{i}})} (C)  \ + \ \sum_{\substack{\textbf{i} \in \mathcal{I}^*: \\ \\ \delta \, \leq \,  \text{Lip}(S_i) \, < \,  \delta^{1-p_{t,\delta}-\varepsilon}}} N_{\delta/\text{Lip}(S_{\textbf{i}})} (C)  \\ \\
&\quad& \qquad \qquad  \ + \   N_\delta \ \Bigg( \  \bigcup_{\textbf{i} \in \mathcal{I}(\delta)} S_{\textbf{i}}(X) \ \Bigg) \ + \ N_\delta(C)\\ \\
&\leq& \sum_{\substack{\textbf{i} \in \mathcal{I}^*: \\ \\ \delta^{1-p_{t,\delta}-\varepsilon} \, \leq \,  \text{Lip}(S_i) \, < \, 1}}   \big(\delta/\text{Lip}(S_{\textbf{i}}) \big)^{-t}  \ + \ \sum_{\substack{\textbf{i} \in \mathcal{I}^*: \\ \\ \delta \, \leq \,  \text{Lip}(S_i) \, < \,  \delta^{1-p_{t,\delta}-\varepsilon}}} N_{\delta^{p_{t,\delta}+\varepsilon}} (C) \\ \\
&\quad& \qquad \qquad   \ + \   \sum_{\textbf{i} \in \mathcal{I}(\delta)} N_{\delta/\text{Lip}(S_{\textbf{i}})} (X)  \ + \ \delta^{-t} \qquad \qquad \text{by (\ref{boundp1}, \ref{boundp2})}\\ \\
&\leq&  \delta^{-t} \sum_{\substack{\textbf{i} \in \mathcal{I}^*: \\ \\ \delta^{1-p_{t,\delta}-\varepsilon} \, \leq \,  \text{Lip}(S_i) \,  < \, 1}} \text{Lip}(S_{\textbf{i}})^{t}  \ + \  \sum_{\substack{\textbf{i} \in \mathcal{I}^*: \\ \\ \delta \,  \leq \, \text{Lip}(S_i) \, < \, \delta^{1-p_{t,\delta}-\varepsilon}}}  \delta^{-(p_{t,\delta}+\varepsilon)t}  \ + \    N_1(X) \, \lvert  \mathcal{I}(\delta) \rvert\\ \\
&\quad& \qquad \qquad \ + \  \delta^{-t} \qquad \qquad  \text{by (\ref{boundp1}, \ref{boundp2})}\\ \\
&\leq&  \delta^{-t} \sum_{\textbf{i} \in \mathcal{I}^*} \text{Lip}(S_{\textbf{i}})^{t}  \ + \  \lvert \{\textbf{i} \in \mathcal{I}^*: \delta \leq \text{Lip}(S_i)   \} \rvert  \  \delta^{-(p_{t,\delta}+\varepsilon)t} \\ \\
&\quad& \qquad \qquad  \ + \    N_1(X) \, \delta^{-s}  \ + \  \delta^{-t} \qquad \qquad  \text{by Lemma \ref{stopsize}} \\ \\
&\leq&\big(  b_t +N_1(X)+1 \big) \, \delta^{-t}  \ + \  \frac{\log \delta}{\log L_{\max}} \, \delta^{-s} \,  \delta^{-(p_{t,\delta}+\varepsilon)t}  \qquad \qquad  \text{by Lemmas \ref{sumsize} and \ref{setsize}} \\ \\
&\leq&\big(  b_t +N_1(X) +1\big) \, \delta^{-t}  \ + \  \frac{\log \delta}{\log L_{\max}}  \, \delta^{-(s+(p_{t}+2\varepsilon)t)}
\end{eqnarray*}
from which it follows that $\underline{\dim}_\text{B} F_C = \underline{\dim}_\text{B} \mathcal{O} \leq \max\{ t, \ s+(p_{t}+2\varepsilon)t\}$ and letting $\varepsilon$ tend to zero yields the desired upper bound.  Note that we do not obtain an upper bound for the upper box dimension here as we only find a sequence of $\delta$s tending to zero for which the above estimate holds.
\end{proof}

\subsection{Proof of Proposition \ref{exlem}}  \label{exlemproof}

 Let $X = [0,1]^d$ for some $d \in \mathbb{N}$ and let $0<b<B\leq d$.  We will first describe a general way of constructing sets $C\subseteq [0,1]^d$ which gives us the required control over the oscillations of the function $N_\delta(C)$.
\\ \\
For $k \in \mathbb{N}$, let $\mathcal{Q}_k$ be the set of closed $2^{-k} \times \cdots  \times 2^{-k}$ cubes formed by imposing a $2^{-k}$ grid on $[0,1]^d$ orientated at the origin.  For each $k$ select a subset of these cubes and call their union $Q_k$.  We assume that $[0,1]^d \supseteq Q_1 \supseteq Q_2 \supseteq \dots$ and that if a cube is chosen at the $k$th step, then at least one sub-cube is chosen at the $(k+1)$th stage.  Finally, we set $C = \cap_{k \in \mathbb{N}} Q_k$.  Let $M_{2^{-k}}(C)$ denote the number of cubes in $\mathcal{Q}_k$ which intersect $C$.  We will only choose cubes at the $k$th level in two different ways:
\\ \\
\emph{Method 1}: at the $(k+1)$th stage we choose \emph{precisely one} cube from each $k$th level cube;
\\ \\
and
\\ \\
\emph{Method 2}: at the $(k+1)$th stage we choose \emph{all} sub-cubes from within each $k$th level cube.
\\ \\
For $\delta \in (0,1)$, let $k(\delta) = \max \big\{k \in \mathbb{N} \cup 0 : \delta \leq 2^{-k} \big\}$.  It is easy to see that
\[
3^{-d} \, M_{2^{-k(\delta)}}(C) \leq N_\delta(C) \leq M_{2^{-(k(\delta)+1+d)}}(C).
\]
Also, for all $k \in \mathbb{N}$,
\[
M_{2^{-k}}(C) \leq M_{2^{-(k+1)}}(C)  \leq 2^d \, M_{2^{-k}}(C)
\]
and these bounds are tight as if at the $(k+1)$th stage we use method 1, then we attain the left hand bound and if at the $(k+1)$th stage we use method 2, then we attain the right hand bound.

\begin{figure}[H]
	\centering
	\includegraphics[width=150mm]{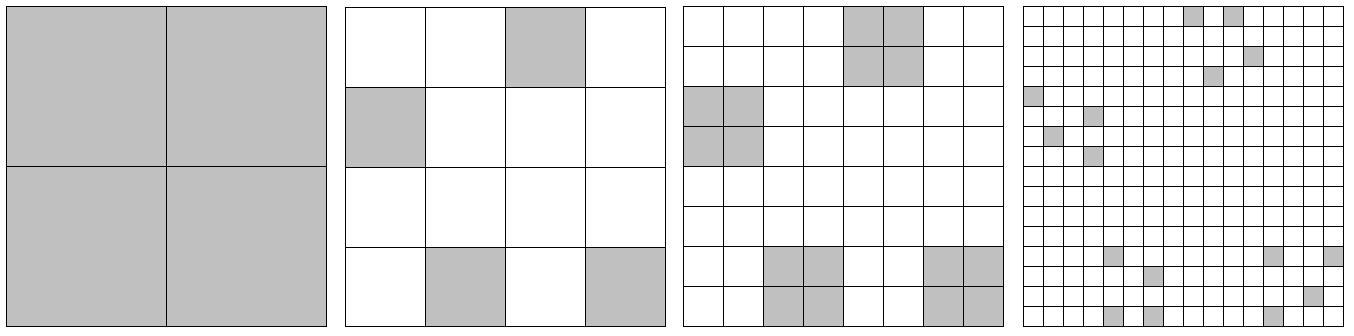}
		\caption{The first 4 steps in the construction of a compact set $C \subset [0,1]^2$ using methods 2, 1, 2, 1 respectively.}
\end{figure}

\emph{Proof of (1):} The key to constructing a compact set $C \subseteq X$ with $p_t(C) = 0$ for all $t\geq b$ is to force $N_\delta(C)$ to be strictly smaller than $\delta^{-b}$ for increasingly long periods of time as $\delta \to 0$.  Let $\mathcal{N}(2,k)$ denote the number of times we use method 2 in the first $k$ steps in the construction of $C$ and let
\[
\overline{\mathcal{N}}(2) = \limsup_{k \to \infty}  \frac{\mathcal{N}(2,k)}{k}
\]
and
\[
\underline{\mathcal{N}}(2) = \liminf_{k \to \infty} \frac{\mathcal{N}(2,k)}{k}.
\]
Observe that
\[
M_{2^{-k}}(C) = 2^{d \mathcal{N}(2,k)}
\]
and hence
\begin{equation} \label{boxdims}
\underline{\dim}_{\text{B}}C = d \,\underline{\mathcal{N}}(2) \qquad \text{and} \qquad  \overline{\dim}_{\text{B}}C = d \,\overline{\mathcal{N}}(2).
\end{equation}
Also observe that if $\delta>0$ is such that $\mathcal{N}(2,k(\delta)+d+1) < bk(\delta)/d$, then
\begin{equation} \label{pestim}
N_{\delta}(C) \leq  M_{2^{-(k(\delta)+1+d)}}(C) = 2^{d \mathcal{N}(2,k(\delta)+d+1)} < 2^{bk(\delta)} \leq \delta^{-b}.
\end{equation}
It is clear that we may alternate between methods 1 and 2 when constructing $C$ in such a way as to ensure that
\[
\overline{\mathcal{N}}(2) = B/d, \,  \qquad \underline{\mathcal{N}}(2) = b/d
\]
and for infinitely many $k_0 \in \mathbb{N}$, we have, for all $k =k_0, \dots, k_0^2$, that
\[
\mathcal{N}(2, k+d+1) < bk/d.
\]
It follows from (\ref{boxdims}) and (\ref{pestim}) that such a compact set $C$ has the desired properties. To show that $p_t(C) = 0$ for all $t \geq b$ it suffices to prove that $p_b(C) = 0$ since $p_t(C)$ is decreasing in $t$ (Lemma \ref{keylem} (2)).  To see that $p_b(C) = 0$ observe that if $\delta>0$ is such that $k(\delta) = k_0^2$ for such a $k_0$ described above, then
\[
N_{\delta'}(C) < (\delta')^{-b}
\]
for all $\delta' \in [\delta, 2^{-k_0}]$ by (\ref{pestim}).  Hence,
\[
(2^{-k_0^2})^{p_{b,\delta}(C)} \geq \delta^{p_{b,\delta}(C)} \geq 2^{-k_0}
\]
which yields $p_{b,\delta}(C) \leq 1/k_0$ and letting $k_0$ tend to infinity (and thus $\delta$ tend to zero) proves that $p_b(C) = 0$. \hfill \qed
\\ \\
\emph{Proof of (2):}   The key to constructing a compact set $C \subseteq X$ with
\[
p_t(C)=\frac{b}{t} \ \frac{d-t}{d-b}
\]
for all $t \in(b,B)$ is to force $N_\delta(C)$ to oscillate as fast as possible as $\delta \to 0$.  We alternate between choosing cubes according to method 1 and 2 as fast as we can making sure that the lower box dimension is $b$ and the upper box dimension is $B$.  Unfortunately, there is a bound on how quickly we can do this (seen in Lemma \ref{keylem} (6)).  We construct $C$ in the following way.  Use method 1 from step 1 until $k_1$ where $k_1 \in \mathbb{N}$ is the first time that
\[
M_{2^{-k_1}}(C) \leq 2^{k_1b}
\]
then change to method 2 from step $k_1+1$ until $k_2>k_1$ where $k_2 \in \mathbb{N}$ is the next occasion where
\[
M_{2^{-k_2}}(C) \geq 3^d \, 2^B \, 2^{k_2B}
\]
then change back to method 1.  Repeat this process as $k \to \infty$ to obtain an infinite increasing sequence $\{k_n\}_{n \in \mathbb{N}}$ where 
\begin{equation} \label{UB}
M_{2^{-k_{2n-1}}}(C) \leq 2^{k_{2n-1}b}
\end{equation}
and
\begin{equation} \label{LB}
M_{2^{-k_{2n}}}(C) \geq 3^d \, 2^{k_{2n}B}
\end{equation}
for each $n \in \mathbb{N}$.  Furthermore, it is clear that 
\[
2^{-b} \,  2^{kb}  \leq M_{2^{-k}}(C) \leq 3^d \, 2^d \, 2^{kB}
\]
for all $k \in \mathbb{N}$ and it follows from this and (\ref{UB}, \ref{LB}) that $b = \underline{\dim}_{\text{B}}C < \overline{\dim}_{\text{B}}C = B$.  Let $t \in (b,B)$ and observe that 
\[
p_t(C) \leq \frac{b}{t} \ \frac{d-t}{d-b}.
\]
by Lemma \ref{keylem} (6).  We will now show the opposite inequality.  For each $k_{2n}$ above, let $\overline{k_{2n}}$ be the biggest integer less than or equal to $B t^{-1} k_{2n}$ and let $\underline{k_{2n}}$ be the smallest integer greater than or equal to $(d-B) (d-t)^{-1} k_{2n}$.  It follows that for each $n \in \mathbb{N}$ we have
\[
N_{2^{-\overline{k_{2n}}}}(C) \geq 3^{-d} \, M_{2^{-\overline{k_{2n}}}}(C) \geq 3^{-d} \, M_{2^{-k_{2n}}}(C)\geq 3^{-d} \, 3^d \, 2^{k_{2n}B}  \geq 2^{\overline{k_{2n}}t} = \big( 2^{-\overline{k_{2n}}}\big)^{-t}
\]
and
\[
N_{2^{-\underline{k_{2n}}}}(C) \geq 3^{-d} \, M_{2^{-\underline{k_{2n}}}}(C) \geq 3^{-d} \,  2^{(\underline{k_{2n}}-k_{2n})d} M_{2^{-k_{2n}}}(C)  \geq 3^{-d} \,  2^{(\underline{k_{2n}}-k_{2n})d}     3^d \, 2^{k_{2n}B} \geq \big( 2^{-\underline{k_{2n}}}\big)^{-t}.
\]
Clearly for $\delta \in (2^{-\overline{k_{2n}}}, 2^{-\underline{k_{2n}}})$ we have $N_\delta(C) \geq \delta^{-t}$.  This implies that, asymptotically, $p_{t,\delta}(C)$ cannot be smaller than the case where $\delta =2^{-\underline{k_{2(n+1)}}}$ and, writing $p = p_{t,2^{-\underline{k_{2(n+1)}}}}(C)$,
\[
2^{-\underline{k_{2(n+1)}}p } = 2^{-\overline{k_{2n}}},
\]
i.e. if $p = \overline{k_{2n}}/{\underline{k_{2(n+1)}}}$.  This yields
\begin{eqnarray}
p_{t}(C) \ \geq \  \liminf_{n \to \infty} \frac{\overline{k_{2n}}}{{\underline{k_{2(n+1)}}}} &\geq& \liminf_{n \to \infty} \frac{k_{2n}}{k_{2(n+1)}} \frac{\Big(B/t-1/k_{2n}\Big)}{\Big((d-B)/(d-t)+1/k_{2(n+1)}\Big)} \nonumber \\ \nonumber \\
&\geq& \frac{B}{t} \ \frac{d-t}{d-B} \  \liminf_{n \to \infty} \frac{k_{2n}}{k_{2(n+1)}} \label{part1}.
\end{eqnarray}
Fix $n \in \mathbb{N}$ and observe that
\[
2^{(k_{2(n+1)}-k_{2n+1})d} \, 2^{k_{2n+1}b-b} \leq 2^{(k_{2(n+1)}-k_{2n+1})d} \, M_{2^{-k_{2n+1}}}(C) \leq  M_{2^{-k_{2(n+1)}}}(C) \leq 3^d \, 2^d \, 2^{k_{2(n+1)}B} \leq 2^{3d+k_{2(n+1)}B}
\]
from which it follows that
\[
(k_{2(n+1)}-k_{2n+1})d +k_{2n+1}b-b \  \leq \  3d+k_{2(n+1)}B
\]
and hence
\begin{equation} \label{part2}
\frac{k_{2n+1}}{k_{2(n+1)}} \ \geq \  \frac{d-B}{d-b} - \frac{b+3d}{k_{2(n+1)}(d-b)}.
\end{equation}
Also we have
\[
2^{-b} \, 2^{(k_{2n+1}-1)b} \leq M_{2^{-(k_{2n+1}-1)}}(C) = M_{2^{-k_{2n}}}(C) \leq 3^d \, 2^d \, 2^{k_{2n}B} \leq 2^{3d+k_{2n}B}
\]
from which it follows that
\[
(k_{2n+1}-1)b-b \ \leq \  3d+k_{2n}B
\]
and hence
\begin{equation} \label{part3}
\frac{k_{2n}}{k_{2n+1}} \ \geq \  \frac{b}{B} - \frac{2b+3d}{k_{2n+1}B}.
\end{equation}
It follows from (\ref{part1}, \ref{part2}, \ref{part3}) that
\[
p_{t}(C) \ \geq \  \frac{B}{t} \ \frac{d-t}{d-B} \  \liminf_{n \to \infty} \frac{k_{2n}}{k_{2(n+1)}} \ \geq \
 \frac{B}{t} \ \frac{d-t}{d-B} \ \frac{b}{B} \ \frac{d-B}{d-b} \ = \ \frac{b}{t} \ \frac{d-t}{d-b}
\]
which is the desired lower bound and completes the proof. \hfill \qed
\\ \\
\begin{centering}

\textbf{Acknowledgements}

\end{centering}

The author was supported by an EPSRC Doctoral Training Grant.

\end{document}